\documentclass[11pt]{article}
\title{On the closure of cyclic subgroups of a free group in pro-$\Vp$ topologies}
\author{}

\usepackage[top=1.5in, bottom=1.5in, left=1in, right=1in]{geometry}
\usepackage{amsmath,amssymb,latexsym,amsthm,enumerate,color}
\usepackage{xypic}
\usepackage{xy}
\usepackage{amstext}
\usepackage{amsbsy}
\usepackage{amsopn}
\usepackage{amsfonts}
\usepackage{eepic}
\usepackage{graphicx}
\usepackage{epsf}
\usepackage{pstricks}

\numberwithin{equation}{section}

\newtheorem{theorem}{Theorem}[section]
\usepackage{graphicx}
\newtheorem{lemma}[theorem]{Lemma}
\newtheorem{proposition}[theorem]{Proposition}
\newtheorem{corollary}[theorem]{Corollary}

\theoremstyle{definition}

\def\Vp{{\rm \mathbf{V}}}

\def\P{\mathbb{P}}
\def\p{\varphi}

\def\N{\mathbb{N}}

\def\ab{{\bf Ab}}

\begin{document}
\author{Claude Marion, Pedro V. Silva, Gareth Tracey}
\maketitle

\begin{center}\small

\bigskip

\end{center}

\abstract{We determine the closure of a cyclic subgroup $H$ of a free group for the pro-$\Vp$ topology when $\Vp$ is an extension-closed pseudovariety of finite groups. We show that $H$ is always closed for the pro-nilpotent topology and compute its closure for the pro-$\mathbf{G}_p$ and pro-$\mathbf{V}_p$ topologies, where $\mathbf{G}_p$ and $\mathbf{V}_p$ denote respectively the pseudovariety of finite $p$-groups and the pseudovariety of finite groups having a normal Sylow $p$-subgroup with quotient an abelian group of exponent dividing $p-1$. More generally, given any nonempty set $P$ of primes, we consider the pseudovariety $\mathbf{G}_P$ of all finite groups having order a product of primes in $P$.
}

\section{Introduction}

The concept of pseudovariety is central in the study of finite algebras. In particular, a {\em pseudovariety of finite groups} is a class of finite groups closed under taking subgroups, homomorphic images and finitary direct products. 
Every pseudovariety $\Vp$ of finite groups induces a (metrizable) topology on any group G, known as the \emph{pro-}$\Vp$ \emph{topology}. The pro-$\Vp$ topology is the initial topology with respect to all homomorphisms $G \to H \in \Vp$, where $H$ is endowed with the discrete topology. 
If ${\bf G}$ denotes the pseudovariety of all finite groups, the pro-${\bf G}$ topology is known as the profinite topology. It was introduced by Marshall Hall in \cite{Hal} and he proved in \cite[Theorem 5.1]{Hal2} that every finitely generated subgroup of a free group is closed for the profinite topology.

Over the years, several other pseudovarieties $\Vp$ were considered \cite{AS,MSW,MSTu,MSTm,MSTs,RZ}. In general, finitely generated subgroups of a free group are not necessarily closed for the pro-$\Vp$ topology. The goal becomes then deciding whether such a subgroup is closed, or even the membership problem for its closure.

In \cite{RZ}, Ribes and Zalesski\v\i \, answered these questions positively for the pseudovariety ${\bf G}_p$ of all finite $p$-groups, for an arbitrary prime $p$. In \cite{MSW}, Margolis, Sapir and Weil dealt successfully with the pseudovariety {\bf N} of all finite nilpotent groups 
(the case of free groups of infinite rank following from Corollary \ref{retres2}).
The pseudovariety {\bf S} of all finite solvable groups has resisted so far all attempts to solve these problems. 
This paper offers a contribution by proving that every cyclic subgroup of a free group is closed for the pro-{\bf N} topology and consequently for the pro-{\bf S} topology.

We note that cyclic subgroups of free groups were considered by Huang, Pawliuk, Sabok and Wise in \cite{HPSW}. 
Let $\mathbb{P}$ be the set of all primes. Given a subset $P$ of primes, let ${\bf G}_P$ denote the pseudovariety of all finite groups whose order is a product of primes in $P$. Note that $\mathbf{G}_{\mathbb{P}}$ is the pseudovariety $\mathbf{G}$ of all finite groups. Let $P^{\perp}$ be the subset of all primes not in $P$. In \cite[Theorem 1.5]{HPSW}, the authors show that if $P$ is a proper subset of primes, $H \leq F_n$ is cyclic and
$$u^p \in H \Rightarrow u \in H$$
holds for all $u \in F_n$ and $p \in P$, then $H$ is ${\bf G}_{P^{\perp}}$-closed
(the converse is proved in Corollary \ref{cycwp}(ii) below).
They also relate this property to the Hrushovsky property (extending
property for partial automorphisms) for hypertournaments.

In Section 2, after fixing some notation and giving some preliminary results, we show that when free groups are residually $\Vp$ we can assume that $F$ is a free group $F_n$ of finite rank $n$. 

In Section 3, we consider the case of extension-closed pseudovarieties $\Vp$. We determine the closure of a cyclic subgroup $H$ of a free group for the pro-$\Vp$ topology and show that it is effectively computable under mild decidability conditions. We show also that all cyclic subgroups being closed is equivalent to $\Vp$ containing all finite abelian groups. We then apply these results to the particular cases of {\bf S} and ${\bf G}_P$.

In Section 4, we use the computation of the ${\bf G}_p$-closure to show that cyclic subgroups are always closed for the pro-{\bf N} topology. This generalizes to any pseudovariety containing {\bf N}, in particular to the pseudovariety ${\bf Su}$ of all finite supersolvable groups. By a theorem of Auinger and Steinberg, ${\bf Su}$ is the join of the pseudovarieties $\Vp_p$ (finite groups having a normal Sylow $p$-subgroup with quotient an abelian group of exponent dividing $p-1$), for every prime $p$. By \cite[Theorem 1.1]{AS}, the pseudovarieties $\Vp_p$ are precisely the maximal {\em Hall subpseudovarieties} of {\bf Su} (see \cite[Subsection 3.5]{AS} for definitions). By \cite{AMV}, $\Vp_p$ is also the class of all finite groups having a faithful representation by upper triangular matrices over the field $\mathbb{F}_p$ of $p$ elements.
We close the paper by proving that the $\Vp_p$-closure of a cyclic subgroup of a free group is effectively computable.

\section{Preliminaries}
\label{prel}

In this paper, $\P$ denotes the set of all primes. Also if $m$ is a positive integer, then $C_m$ denotes the cyclic group of order $m$.

If $A$ is a set, we set $F_A$ to be the free group on $A$. Also if $H$ is a finitely generated subgroup of $F_A$, we write $H\leq_{f.g.} F_A$.
Let $F_n$ denote the free group of rank $n \in \N$ and fix a basis $A=\{a_1,\dots,a_n\}$ of $F_n$. 
We consider the $n$ homomorphisms $f_{a_i}: F_n\rightarrow \mathbb{Z}$ where $1\leq i \leq n$ defined by 
$$f_{a_i}(a_j)=\left\{
\begin{array}{ll}
1 & \textrm{if} \ j=i\\
0 & \textrm{otherwise}
\end{array}
\right.
$$   
for $1\leq j\leq n$.

We call a pseudovariety of finite groups {\em nontrivial} if it contains some nontrivial group. Besides the pseudovarieties defined in the Introduction, we consider also, for each $m \geq 1$, the pseudovariety $\ab(m)$ of all finite abelian groups of exponent dividing $m$. This pseudovariety is an {\em equational} pseudovariety because it is defined by the identities $[x,y] = 1$ and $x^m = 1$. This amounts to saying that a group $G$ is in $\ab(m)$ if and only if $[g,h] = g^m = 1$ for all $g,h \in G$. 

For all $m,n \geq 1$, write
$$F_n^{\ab(m)} = \langle [u,v],u^m \mid u,v \in F_n \rangle \unlhd F_n.$$
It is easy to check that 
$$F_n/F_n^{\ab(m)} \cong C_m^n$$
is the free object of $\ab(m)$ on the set $A$.

Following \cite{AS}, we may use the semidirect product of pseudovarieties to get the decomposition $\Vp_p = {\bf G}_p \ast \ab(p-1)$. This alternative description of $\Vp_p$ follows from the Kalu\v{z}nin-Krasner theorem \cite{Neu}.

A finite group is {\em supersolvable} if each of its chief factors is cyclic of prime order. 
By \cite[Corollary 2.7]{AS},  we have $$\mathbf{Su}=\bigvee_{p\in \mathbb{P}} \mathbf{V}_p.$$

We say that a pseudovariety $\Vp$ of finite groups is {\em extension-closed} if 
$$N,G/N \in \Vp \Rightarrow G \in \Vp$$
holds for every (finite) group $G$ and every $N \unlhd G$. 

It is well known that $\mathbf{S}$ is extension-closed, and it is obvious that $\mathbf{G}_P$ is extension-closed for every $P \subseteq \mathbb{P}$.

The smallest non nilpotent group is the symmetric group $S_3$ (which is not simple), hence $\mathbf{N}$ is not extension-closed. On the other hand, the alternating group $A_4$ is the smallest non supersolvable group, hence ${\bf Su}$ is not extension-closed either.

By considering the cyclic group $C_{m^2}$, we see that $\ab(m)$ is not extension-closed for every $m \geq 2$, and $C_{(p-1)^2}$ shows that $\mathbf{V}_p$ is not extension-closed either for every prime $p > 2$.

Given a pseudovariety $\mathbf{V}$ of finite groups, where we consider finite groups endowed with the discrete topology, the pro-$\mathbf{V}$ topology on a group $G$ is defined as the coarsest topology which makes all morphisms from $G$ into elements of $\mathbf{V}$ continuous. Equivalently, $G$ is a topological group where the normal subgroups $K$ of $G$ such that $G/K\in \mathbf{V}$ form a basis of neighbourhoods of the identity. 

For a topological property $\mathcal{P}$ and a subset $S$ of a group $G$, we say that $S$ is $\mathbf{V}$-$\mathcal{P}$ if $S$ has property $\mathcal{P}$ in the pro-$\mathbf{V}$ topology on $G$. 

Given a subgroup $H$ of a group $G$, the core $\textrm{Core}_G(H)$ of $H$ in $G$ is the largest normal subgroup of $G$ contained in $H$ and is equal to $\bigcap_{g\in G} g^{-1}Hg$.

\begin{theorem}
\label{sopen}
{\rm \cite[Proposition 1.2]{MSW}}
Given a group $G$ and $H \leq G$, the following conditions are equivalent:
\begin{itemize}
\item[(i)] $H$ is $\mathbf{V}$-open;
\item[(ii)] $H$ is $\mathbf{V}$-clopen;
\item[(iii)] $G/\mathrm{Core}_G(H) \in \mathbf{V}$.
\end{itemize}
\end{theorem}

It is easy to prove (see \cite[Proposition 2.6]{MSTm}) the following:
\begin{equation}
\label{fioc}
\mbox{If $[G:H] < \infty$, then $H$ is $\mathbf{V}$-open if and only if $H$ is $\mathbf{V}$-closed.}
\end{equation}

We note that a subgroup $H$ of $G$ is $\mathbf{V}$-closed if and only if, for every $g \in G\setminus H$, there exists some $\Vp$-clopen $K \leq G$ such that $H \leq K$ and $g \notin K$.
On the other hand, a subgroup $H$ of $G$ is $\mathbf{V}$-dense if and only if $HN=G$ for every normal subgroup $N$ of $G$ such that $G/N\in \mathbf{V}$.

Given a group $G$ and $S \subseteq G$, we denote by  ${\rm{Cl}}_{\mathbf{V}}^G(S)$ the $\mathbf{V}$-closure of $S$ in $G$. We can omit the superscript if $G$ is clear from the context.

Note that if {\bf V} and {\bf W} are pseudovarieties of finite groups, then
\begin{equation}
\label{inccl}
{\bf W} \subseteq {\bf V} \Rightarrow {\rm{Cl}}_{\mathbf{V}}^G(H) \leq {\rm{Cl}}_{\mathbf{W}}^G(H)
\end{equation}
holds for every subgroup $H$ of a group $G$ by \cite[Corollary 3.2]{MSW}. 

A subgroup $K$ of a group $G$ is called a {\em retract} if there exists a (surjective) homomorphism $\p:G \to K$ fixing the elements of $K$. If $\Vp$ is a pseudovariety of finite groups, the group $G$ is said to be {\em residually} $\Vp$ if $\langle 1\rangle$ is a $\Vp$-closed subgroup of $G$. The following result is well known, but we include a proof for the sake of completeness:

\begin{lemma}
\label{rvec}
Every free group is residually $\Vp$ for every nontrivial extension-closed pseudovariety $\Vp$ of finite groups.
\end{lemma}

\begin{proof}
Since $\Vp$ is nontrivial, then $C_p \in \Vp$ for some prime $p$. Since every finite $p$-group admits a composition series where the factor groups are cyclic of order $p$, it follows from $\Vp$ being extension-closed that ${\bf G}_p \subseteq \Vp$.

By  \cite[Theorem 6]{Tak}, 
every free group is residually  $\mathbf{G}_p$, and therefore residually $\Vp$ in view of (\ref{inccl}).
\end{proof}

Now we can prove the following:

\begin{proposition}
\label{retres}
Let $\Vp$ be a pseudovariety of finite groups. Let $G$ be a residually $\Vp$ group and let $K \leq G$ be a retract of $G$. Let $H \leq K$. Then ${\rm Cl}_{\mathbf{V}}^G(H)={\rm Cl}_{\mathbf{V}}^K(H)$. 
\end{proposition}

\begin{proof}
Write $C_1 = {\rm Cl}_{\mathbf{V}}^G(H)$ and $C_2 = {\rm Cl}_{\mathbf{V}}^K(H)$. By \cite[Proposition 1.6]{MSW}, the pro-$\Vp$ topology of $K$ coincides with the subspace topology of $K$ with respect to the pro-$\Vp$ topology of $G$. Hence $C_2 = C_1\cap K$. Since $G$ is residually $\Vp$, the retract $K$ is $\Vp$-closed in $G$ by \cite[Corollary 1.8]{MSW}. Hence $C_1=C_1\cap  K=C_2$.
\end{proof}

\begin{corollary}
\label{retres2}
Let $H$ be a finitely generated subgroup of a free group $F_A$ of arbitrary rank. Let $B \subseteq A$ be a finite set such that $H \leq F_B$. Let $\mathbf{V}$ be a pseudovariety of finite groups for which $F_A$ is residually $\mathbf{V}$. Then:
\begin{itemize}
\item[(i)] ${\rm Cl}_{\mathbf{V}}^{F_A}(H)={\rm Cl}_{\mathbf{V}}^{F_B}(H)$;
\item[(ii)] $H$ is $\mathbf{V}$-closed in $F_A$ if and only if  $H$ is $\mathbf{V}$-closed in $F_B$;
\item[(iii)] if $A$ is infinite then $H$ is not $\mathbf{V}$-dense in $F_A$. 
\end{itemize}
This applies in particular to 
{\bf N}, {\bf Su}, $\mathbf{V}_p$ (for every prime $p$) and every nontrivial extension-closed pseudovariety (including 
{\bf S} and $\mathbf{G}_P$ (for every nonempty $P \subseteq \mathbb{P}$)).
\end{corollary}

\begin{proof}
(i) follows from Proposition \ref{retres}, since $F_B$ is a retract of $F_A$. Now (ii) and (iii) follow from (i).

By Lemma \ref{rvec}, this applies to every nontrivial extension-closed pseudovariety. Since the remaining pseudovarieties contain $\mathbf{G}_p$ for some prime $p$, it follows from (\ref{inccl}) that they inherit the same property.
\end{proof}

As a consequence, we may always assume that free groups have finite rank in our proofs.

\section{Extension-closed pseudovarieties}

We discuss in this section the extension-closed case.
The following result combines some of the special properties shared by extension-closed pseudovarieties:

\begin{theorem}
\label{ecp}
{\rm \cite[Proposition 2.17]{MSW}, \cite[Corollary 3.3 and Proposition 3.4]{RZ}}
Let $\Vp$ be a nontrivial extension-closed pseudovariety of finite groups and let $F$ be a free group. Let $H \leq_{f.g.} F$. Then:
\begin{itemize}
\item[(i)]
$H$ is $\Vp$-closed if and only if $H$ is a free factor of a $\Vp$-clopen subgroup of $F$;
\item[(ii)]
the rank of ${\rm{Cl}}_{\mathbf{V}}(H)$ is at most the rank of $H$;
\item[(iii)]
the pro-$\mathbf{V}$ topology on ${\rm{Cl}}_{\mathbf{V}}(H)$ is the subspace topology with respect to the pro-$\mathbf{V}$ topology on $F$. 
\end{itemize}
\end{theorem}

We define  the {\em root} and the  {\em exponent} (in $F$) of a word $w\in F\setminus \{1\}$ as follows. Clearly, $w$ can be written as a power $u^r$ for some $u \in F$ and $r \geq 1$, and there exist finitely many such decompositions. The maximum such $r$ is the exponent of $w$ and is denoted by ${\rm{exp}}(w)$. Then the root of $w$ is defined through the equality $w = {\rm{root}}(w)^{{\rm{exp}}(w)}$. A classical theorem states that two nonempty words $x,y \in F$ commute if and only if $x,y$ are powers of a same word if and only if $x^m = y^k$ for some nonzero $m,k \in \mathbb{Z}$. It follows that, for all $w\in F\setminus \{1\}$ and $s\in \mathbb{Z}\setminus\{0\}$,

\begin{equation}
\label{pow}
\mbox{$w = v^s$ implies $v = ({\rm{root}}(w))^k$ for some $k \in \mathbb{Z} \setminus \{ 0\}$.}
\end{equation}

From the same theorem we may deduce that, for every $s\in \mathbb{Z}\setminus\{0\}$,
\begin{equation}
\label{pow1}
\mbox{
${\rm{root}}(w^s)$ = $\left\{
\begin{array}{ll}
{\rm{root}}(w) & \textrm{if} \ s> 0\\
{\rm{root}}(w)^{-1} & \textrm{if} \ s< 0
\end{array}
\right.
$
 and ${\rm{exp}}(w^s) = |s|{\rm{exp}}(w)$.}
\end{equation}

Indeed, since 
$$({\rm{root}}(w))^{s{\rm{exp}}(w)} = w^s = ({\rm{root}}(w^s))^{{\rm{exp}}(w^s)},$$
it follows that ${\rm{root}}(w)$ and ${\rm{root}}(w^s)$ are powers of a same word and so by maximality of exponents ${\rm{root}}(w^s) = {\rm{root}}(w)^{\pm 1}$: Moreover, ${\rm{root}}(w^s)={\rm{root}}(w)$ if and only if $s>0$. Hence
$$({\rm{root}}(w))^{s{\rm{exp}}(w)} =\left\{
\begin{array}{ll}
 ({\rm{root}}(w))^{{\rm{exp}}(w^s)} & \textrm{if} \ s>0\\
 ({\rm{root}}(w)^{-1})^{{\rm{exp}}(w^s)} & \textrm{if} \ s<0
 \end{array}
 \right.
 $$
and so (\ref{pow1}) holds. In \cite{BF} Berlai and Ferov generalize the concepts of root and exponent to right-angled Artin groups.

\begin{theorem}
\label{eccy}
Let $\Vp$ be a nontrivial extension-closed pseudovariety of finite groups and let $F$ be a free group. Let $w \in F\setminus \{ 1\}$, $u = {\rm{root}}(w)$ and $e = {\rm exp}(w)$. Then:
\begin{itemize}
\item[(i)]
${\rm{Cl}}_{\mathbf{V}}(\langle 1\rangle) = \langle 1\rangle$;
\item[(ii)]
${\rm{Cl}}_{\mathbf{V}}(\langle w\rangle) \leq \langle u\rangle$;
\item[(iii)]
$\langle w\rangle$ is $\Vp$-closed if and only if $C_e \in \Vp$;
\item[(iv)]
${\rm{Cl}}_{\mathbf{V}}(\langle w\rangle) = \langle u^m\rangle$ for 
$$m = \max\{ k \geq 1 \mid k \mbox{ divides $e$ and }C_k \in \Vp\}.$$
\item[(v)]
if we can decide, for each $n \geq 2$, whether or not $C_n \in \Vp$, then ${\rm{Cl}}_{\mathbf{V}}(\langle w\rangle)$ is effectively computable.
\end{itemize}
\end{theorem}

\begin{proof}
(i) By Lemma \ref{rvec}.

(ii) Write $C = {\rm{Cl}}_{\mathbf{V}}(\langle w\rangle)$. By Theorem \ref{ecp}(ii), we have $C = \langle v \rangle$ for some $v \in F$. Hence $w = v^r$ for some $r \in \mathbb{Z}$ and so $v \in \langle u\rangle$ by (\ref{pow}). Thus $C \leq \langle u\rangle$. 

(iii) Since ${\rm{root}}(u) = {\rm{root}}(u^e) = u$ by (\ref{pow1}), it follows from part (ii) that $\langle u\rangle$ is $\Vp$-closed. By Theorem \ref{ecp}(iii), $\langle w\rangle$ is $\Vp$-closed in $F$ if and only if it is $\Vp$-closed in $\langle u\rangle$. Since $\langle w\rangle$ is a finite index (normal) subgroup of $\langle u\rangle$, it follows from Theorem \ref{sopen} and (\ref{fioc}) that $\langle w\rangle$ is $\Vp$-closed in $F$ if and only if 
$\langle u\rangle/\langle w\rangle \in \Vp$. Since $\langle u\rangle/\langle w\rangle = \langle u\rangle/\langle u^e\rangle \cong C_e$, we get the desired equivalence.

(iv) By parts (ii) and (iii), ${\rm{Cl}}_{\mathbf{V}}(\langle w\rangle) = \langle u^m\rangle$ for some $m \geq 1$ such that  $C_m \in \Vp$. Since $w = u^e \in \langle u^m\rangle$, then $m$ divides $e$. Suppose that $k \geq 1$ divides $e$ and $C_k \in \Vp$. By part (iii), $\langle u^k\rangle$ is $\Vp$-closed, and we have $w = u^e \in \langle u^k\rangle$. Hence $\langle w\rangle \leq \langle u^k\rangle$ and by minimality of the closure we get $\langle u^m\rangle \leq \langle u^k\rangle$. Thus $k$ divides $m$ and so $k \leq m$ as required.

(v) It follows immediately from part (iv).
\end{proof}

\begin{corollary}
\label{eccyc}
Let $\Vp$ be an extension-closed pseudovariety of finite groups. Then the following conditions are equivalent: 
\begin{itemize}
\item[(i)]
every cyclic subgroup of a free group is $\Vp$-closed;
\item[(ii)]
$\ab \subseteq \Vp$.
\end{itemize}
\end{corollary}

\begin{proof}
(i) $\Rightarrow$ (ii). In view of the fundamental theorems of finite abelian groups, it suffices to show that $C_n \in \Vp$  for every $n \geq 2$. Let $u$ be a primitive word of $F$. Then ${\rm{root}}(u) = u$ and ${\rm exp}(u^n) = n$ by (\ref{pow1}). Since $\langle u^n \rangle$ is $\Vp$-closed, it follows from Theorem \ref{eccy}(iii) that $C_n \in \Vp$ and so $\ab \subseteq \Vp$.

(ii) $\Rightarrow$ (i). By Theorem \ref{eccy}(i) and (iii).
\end{proof}

Since it is well known that the pseudovariety ${\bf S}$ of all finite solvable groups is extension-closed, Corollary \ref{eccyc} and Theorem \ref{ecp}(i) yield:

\begin{corollary}
\label{t:cyclics}
Let $H$ be a cyclic subgroup of a free group $F$. Then:
\begin{itemize}
\item[(i)] $H$ is $\mathbf{S}$-closed;
\item[(ii)] $H$ is a free factor of an $\mathbf{S}$-clopen subgroup of $F$. 
\end{itemize}
\end{corollary}

Given $P \subseteq \mathbb{P}$ nonempty and $k \geq 1$, let
$\nu_P(k)$ be the largest divisor of $k$ which is a product of primes in $P$.
Since $\mathbf{G}_P$ is extension-closed, we can now deduce the following result:

\begin{corollary}
\label{cycwp}
Let $P \subseteq \mathbb{P}$ be nonempty and let $F$ be a free group. Let $w \in F\setminus \{ 1\}$, $u = {\rm{root}}(w)$ and $e = {\rm exp}(w)$. Then:
\begin{itemize}
\item[(i)]
${\rm{Cl}}_{\mathbf{G}_P}(\langle 1\rangle) = \langle 1\rangle$;
\item[(ii)]
$\langle w\rangle$ is $\mathbf{G}_P$-closed if and only if $e$ is a product of primes in $P$
if and only if
\begin{equation}
\label{wetal}
v^p \in \langle w\rangle \Rightarrow v \in \langle w\rangle
\end{equation}
holds for all $v \in F$ and $p \in P^{\perp}$;
\item[(iii)]
${\rm{Cl}}_{\mathbf{G}_P}(\langle w\rangle) = \langle u^{\nu_P(e)}\rangle$. 
\end{itemize}
\end{corollary}

\begin{proof}
(i) By Lemma \ref{rvec}.

(ii) By Theorem \ref{eccy}(iii), $\langle w\rangle$ is $\mathbf{G}_P$-closed if and only if $C_e \in {\bf G}_P$, that is, $e$ is a product of primes in $P$.

Suppose that some $p \in P^{\perp}$ divides $e$. Then (\ref{wetal}) fails for $p$ and $v = u^{\frac{e}{p}}$. The converse implication follows from \cite[Theorem 1.5]{HPSW}, but we include a proof for the sake of completeness. 

Assume that $e$ is a product of primes in $P$. Let $v \in F$ and $p \in P^{\perp}$ be such that $v^p \in \langle w\rangle$. By (\ref{pow1}), we get ${\rm{root}}(v) = {\rm{root}}(v^p) = {\rm{root}}(w)^{\pm 1} = u^{\pm 1}$, hence $v^e = w^{\pm 1} \in \langle w\rangle$. Now ${\rm gcd}(p,e) = 1$ yields $1 = pr + es$ for some $r,s \in \mathbb{Z}$ and so
$v = v^{pr+es} = (v^p)^r(v^e)^s \in \langle w\rangle$. Therefore (\ref{wetal}) holds as required.

(iii) By Theorem \ref{eccy}(iv), it suffices to note that $\nu_P(e)$ is the maximum divisor $k$ of $e$ such that $C_k \in {\bf G}_P$.
\end{proof}

By considering a single prime, we get:

\begin{corollary}
\label{cycgp}
Let $p \in \mathbb{P}$ and let $F$ be a free group. Let $w \in F\setminus \{ 1\}$, $u = {\rm{root}}(w)$ and $e = {\rm exp}(w)$. Then:
\begin{itemize}
\item[(i)]
${\rm{Cl}}_{\mathbf{G}_p}(\langle 1\rangle) = \langle 1\rangle$;
\item[(ii)]
$\langle w\rangle$ is $\mathbf{G}_p$-closed if and only if $e = p^s$ for some $s \geq 0$ if and only if
$v^q \in \langle w\rangle \Rightarrow v \in \langle w\rangle$ holds for all $v \in F$ and $q \in \{ p\}^{\perp}$;
\item[(iii)]
${\rm{Cl}}_{\mathbf{G}_p}(\langle w\rangle) = \langle u^{\nu_p(e)}\rangle$. 
\end{itemize}
\end{corollary}

In the language of \cite{BF}, Corollary \ref{cycgp}(ii) is equivalent to saying that given a prime $p$, a cyclic subgroup of a free group is pro-$p$ closed if and only if it is a $p$-isolated subgroup. More generally in \cite{BF} it is shown that given a prime $p$, a $p$-isolated cyclic subgroup of a right-angled Artin group is pro-$p$ closed.

The pseudovariety ${\bf O}$ of all (finite) groups of odd order is equal to ${\bf G}_{2^{\perp}}$, thus we also have:

\begin{corollary}
\label{cyco}
Let $F$ be a free group. Let $w \in F\setminus \{ 1\}$, $u = {\rm{root}}(w)$ and $e = {\rm exp}(w)$. Then:
\begin{itemize}
\item[(i)]
${\rm{Cl}}_{\mathbf{O}}(\langle 1\rangle) = \langle 1\rangle$;
\item[(ii)]
$\langle w\rangle$ is $\mathbf{O}$-closed if and only if $e$ is odd if and only if
$v^2 \in \langle w\rangle \Rightarrow v \in \langle w\rangle$ holds for every $v \in F$;
\item[(iii)]
${\rm{Cl}}_{\mathbf{O}}(\langle w\rangle) = \langle u^{\nu_{\{2\}^{\perp}}(e)}\rangle = \langle u^{\frac{e}{\nu_2(e)}}\rangle$. 
\end{itemize}
\end{corollary}

\section{Non extension-closed pseudovarieties of solvable groups}

We start by considering one of the most important cases: nilpotent groups.

\begin{theorem}
\label{t:cyclicn}
Every cyclic subgroup of a free group is $\mathbf{N}$-closed.
\end{theorem}

\begin{proof}
Let $F$ be a free group 
and let $w \in F$. By \cite[Corollary 4.2(2)]{MSW}, we have
$${\rm Cl}_{\mathbf{N}}(\langle w\rangle) = \bigcap_{p\in \P} {\rm{Cl}}_{\mathbf{G}_p}(\langle w\rangle).$$
If $w = 1$, then Corollary \ref{cycgp}(i) implies that $\langle 1\rangle$ is $\mathbf{N}$-closed, hence we may assume that $w \neq 1$. Let $u = {\rm root}(w)$ and $e = {\rm exp}(e)$. Write $e = p_1^{s_1}\ldots p_m^{s_m}$ with $p_1,\ldots,p_m$ distinct primes and $s_1,\ldots,s_m \geq 1$. 

Let $v \in {\rm Cl}_{\mathbf{N}}(\langle w\rangle)$. Then $v \in {\rm Cl}_{\mathbf{G}_{p_i}}(\langle w\rangle)$ for $i = 1,\ldots,s$. It follows from Corollary \ref{cycgp}(iii) that 
$$v \in \langle u^{\nu_{p_i}(e)} \rangle = \langle u^{p_i^{s_i}} \rangle$$
for $i = 1,\ldots,s$. Since the primes $p_i$ are distinct, we get
$$v \in \langle u^{p_1^{s_1}\ldots p_m^{s_m}} \rangle = \langle u^{e} \rangle = \langle w \rangle.$$
Therefore $\langle w\rangle$ is $\mathbf{N}$-closed.
\end{proof}

\begin{corollary}
\label{containsN}
Let $\Vp$ be a pseudovariety of finite groups containing {\bf N}. Let $H$ be a cyclic subgroup of a free group $F$. Then:
\begin{itemize}
\item[(i)] $H$ is $\Vp$-closed;
\item[(ii)] $H$ is a free factor of a $\mathbf{V}$-clopen subgroup of $F$. 
\end{itemize}
\end{corollary}

\begin{proof}
(i) By Theorem \ref{t:cyclicn} and (\ref{inccl}).

(ii) By Theorem \ref{ecp}(i).
\end{proof}

In particular, this applies to the pseudovariety {\bf Su} (and provides an alternative proof to Corollary \ref{t:cyclics}. Since 
$$\mathbf{Su}=\bigvee_{p\in \mathbb{P}} \mathbf{V}_p,$$
we concentrate now our efforts on the pseudovariety $\mathbf{V}_p$, for a given prime $p$ which we assume fixed throughout the rest of the paper.
In view of Corollary \ref{retres2}, we may assume that $F$ has finite rank $n$. Write $K_n = F_n^{\ab(p-1)} \unlhd F_n$. 

We can derive from the proof of \cite[Proposition 4.4]{AS} the following result:

\begin{proposition}
\label{vtog}
Let $H \leq_{f.g.} F_n$ and let $p \in \mathbb{P}$.
The following conditions are equivalent:
\begin{enumerate}[(i)]
\item[(i)] $H$ is $\mathbf{V}_p$-closed in $F_n$;
\item[(ii)] $H \cap K_n$ is $\mathbf{G}_p$-closed in $K_n$.
\end{enumerate}
\end{proposition}

\begin{proof}
Write $F = F_n$ and $K = K_n$.
Since $[F:K] < \infty$ and $K \unlhd F$, the second isomorphism theorem yields $[H:H\cap K] < \infty$ as well. Hence we may write
\begin{equation}
\label{vtog1}
H = \bigcup_{i=0}^m (H\cap K)t_i
\end{equation}
for some $t_0,\ldots,t_m \in F$ with $t_0 = 1$. 

By the proof of \cite[Proposition 4.4]{AS}, we know that
$$\textrm{Cl}_{\mathbf{V}_p}^F(H) = \bigcup_{i=0}^m (\textrm{Cl}_{\mathbf{V}_p}^F(H\cap K))t_i.$$
The authors prove this statement using the concept of a Schreier transversal, but we do not need to worry about that because Schreier transversals do always exist. They also prove that
$$\textrm{Cl}_{\mathbf{V}_p}^F(H\cap K) = \textrm{Cl}_{\mathbf{G}_p}^K(H\cap K).$$
It follows that 
\begin{equation}
\label{vtog2}
\textrm{Cl}_{\mathbf{V}_p}^F(H) = \bigcup_{i=0}^m (\textrm{Cl}_{\mathbf{G}_p}^K(H\cap K))t_i.
\end{equation}
Now we prove the stated equivalence.

Assume first that $H \cap K$ is $\mathbf{G}_p$-closed in $K$. In other words $\textrm{Cl}_{\mathbf{G}_p}^K(H\cap K) = H\cap K$, and so it follows from (\ref{vtog1}) and (\ref{vtog2}) that
$$\textrm{Cl}_{\mathbf{V}_p}^F(H) = \bigcup_{i=0}^m (\textrm{Cl}_{\mathbf{G}_p}^K(H\cap K))t_i = \bigcup_{i=0}^m (H\cap K)t_i = H.$$
Thus $H$ is $\mathbf{V}_p$-closed in $F$.

Conversely, assume that $H$ is $\mathbf{V}_p$-closed in $F$, that is, $\textrm{Cl}_{\mathbf{V}_p}^F(H) = H$. Then in view of (\ref{vtog2}) we get
$$H = \bigcup_{i=0}^m (\textrm{Cl}_{\mathbf{G}_p}^K(H\cap K))t_i.$$
Since $t_0 = 1$, this implies $\textrm{Cl}_{\mathbf{G}_p}^K(H\cap K) \subseteq H$ and consequently $\textrm{Cl}_{\mathbf{G}_p}^K(H\cap K) \subseteq H\cap K$. Thus $\textrm{Cl}_{\mathbf{G}_p}^K(H\cap K) = H\cap K$ and $H \cap K$ is $\mathbf{G}_p$-closed in $K$ as required.
\end{proof}

In view of Corollary \ref{retres2}, $\langle1\rangle$ is $\mathbf{V}_p$-closed in $F_n$.
For a nontrivial cyclic subgroup of $F_n$, more work is needed in order to determine its $\mathbf{V}_p$-closure.
With respect to the basis $A = \{ a_1,\ldots,a_n\}$ of $F_n$, we recover the homomorphisms $f_{a_i}:F_n \to \mathbb{Z}$ defined in the beginning of Section \ref{prel}. 
Given $w\in F_n$, write
$$h_w = \frac{p-1}{{\rm gcd}(p-1,f_{a_1}(w),\ldots,f_{a_n}(w))}.$$

\begin{lemma}
\label{pk}
Let $w \in F_n$. Then
$h_w = {\rm min}\{ r \geq 1 \mid w^r \in K_n\}$.
\end{lemma}

\begin{proof}
Since $K_n = F_n^{\ab(p-1)} = \langle [u,v],u^{p-1} \mid u,v \in F_n \rangle$, we have 
$$\begin{array}{lll}
w^r \in K_n&
\Leftrightarrow& \forall i \in \{ 1,\ldots,n\}\, p-1 \mid f_{a_i}(w^r)\\ &&\\
&\Leftrightarrow&p-1 \mid {\rm{gcd}}(f_{a_1}(w^r),\ldots,f_{a_n}(w^r)) \\
&\Leftrightarrow &p-1 \mid r\,{\rm{gcd}}(f_{a_1}(w),\ldots,f_{a_n}(w))\\ &&\\
&\Leftrightarrow&\frac{p-1}{{\rm{gcd}}(p-1,f_{a_1}(w),\ldots,f_{a_n}(w))} \mid r\\
&\Leftrightarrow& h_w \mid r
\end{array}$$
and the claim follows.
\end{proof}

Given $w \in K_n$ nonempty, we now characterise the root and the exponent in $K_n$ of $w$ in terms of the root and the exponent in $F_n$ of $w$.

\begin{lemma}\label{rek}
Let $w \in K_n \leq F_n$ be nonempty. Let $u$ and $e$ denote respectively the root and exponent of $w$ in $F_n$. Let $u_K$ and $e_K$ denote respectively the root and exponent of $w$ in $K_n$. 
Then:
\begin{enumerate}[(i)]
\item $u_K = u^{h_u}$;
\item $e_K = \frac{e}{h_u}$;
\item $u_K^{\nu_p(e_K)} = u^{h_u\nu_p(e)}$.
\end{enumerate}
\end{lemma}

\begin{proof}
Since $u_K^{e_K} = w$, it follows from (\ref{pow}) that $u_K = u^q$ for some positive integer $q$. 
Moreover, $e_K = \frac{e}{q}$. Now $u_K \in K_n$ and we must have in fact
\begin{equation}
\label{rek1}
q = \min\{ r \geq 1 \mid u^r \in K_n\mbox{ and } r\mid e\}.
\end{equation}
But $u^e = w \in K_n$, hence it follows from the proof of Lemma \ref{pk} that $h_u \mid e$. Now Lemma \ref{pk} and (\ref{rek1}) together yield $q = h_u$. Thus (i) and (ii) hold.

Finally, to prove (iii) we remark that $h_u \mid p-1$ yields
$$u_K^{\nu_p(e_K)} = (u^{h_u})^{\nu_p(\frac{e}{h_u})} = (u^{h_u})^{\nu_p(e)}
= u^{h_u\nu_p(e)}.$$
\end{proof}

We now prove the following criterion for a cyclic subgroup of $F_n$ to be $\mathbf{V}_p$-closed.

\begin{proposition}
\label{cyv}
Let $w\in F_n \setminus \{ 1\}$. Let $u = {\rm root}(w)$ and $e = {\rm exp}(w)$.
Then $h_w \mid h_u$ and the following conditions are equivalent:
\begin{enumerate}[(i)]
\item $\langle w\rangle$ is $\mathbf{V}_p$-closed in $F_n$;
\item $e = \frac{h_u}{h_w} p^s$ for some $s \geq 0$.
\end{enumerate}
\end{proposition}

\begin{proof}
Write $F = F_n$ and $K = K_n$. We have
$$h_w \mid h_u \Leftrightarrow {\rm{gcd}}(p-1,f_{a_1}(u),\ldots, f_{a_n}(u)) \mid {\rm{gcd}}(p-1,f_{a_1}(w),\ldots, f_{a_n}(w)).$$
Since $f_{a_i}(w) = f_{a_i}(u^e) = ef_{a_i}(u)$ for $i = 1,\ldots,n$, it follows that $h_w \mid h_u$.

Let $H = \langle w\rangle$. By Proposition \ref{vtog}, $H$ is $\mathbf{V}_p$-closed in $F$ if and only if $H \cap K$ is $\mathbf{G}_p$-closed in $K$. 

Since $H \cap K$ is a subgroup of a cyclic group, it is itself cyclic, and it follows from Lemma \ref{pk} that 
\begin{equation}
\label{cyv2}
H \cap K = \langle w^{h_w} \rangle.
\end{equation}

Now we have to find out when $H\cap K$ is $\mathbf{G}_p$-closed in $K$. By Corollary \ref{cycgp}(ii), this is equivalent to having 
$${\rm exp}_K(w^{h_w}) = p^s$$
for some $s \geq 0$. 
By Lemma \ref{rek}(ii), this is equivalent to having
$${\rm exp}(w^{h_w}) = h_up^s$$
for some $s \geq 0$. Since ${\rm exp}(w^{h_w}) = h_w{\rm exp}(w)$ by (\ref{pow1}), the result follows.
\end{proof}

Finally we determine the $\mathbf{V}_p$-closure of a nontrivial cyclic subgroup of $F_n$.

\begin{corollary}
\label{clv}
Let $w \in F_n \setminus \{ 1\}$. Let $u = {\rm root}(w)$ and $e = {\rm exp}(w)$. Let $d_p = {\rm gcd}(e,h_u)\nu_p(e)$.
Then 
${\rm{Cl}}_{\mathbf{V}_p}(\langle w\rangle) = 
\langle u^{d_p}\rangle$.
\end{corollary}

\begin{proof}
Write $F = F_n$ and $K = K_n$. Write also $p^s = \nu_p(e)$
and $h = h_w$.
By (\ref{cyv2}), we have $H\cap K = \langle w^h\rangle$. Since
$$\langle w\rangle = \bigcup_{i=0}^{h-1} \langle w^h\rangle w^i,$$
it follows from (\ref{vtog2}) that
\begin{equation}
\label{clv1}
{\rm{Cl}}_{\mathbf{V}_p}^F(\langle w\rangle) = \bigcup_{i=0}^{h-1} ({\rm{Cl}}_{\mathbf{G}_p}^K(\langle w^h\rangle))w^i.
\end{equation}
Now Corollary \ref{cycgp}(iii) yields
$${\rm{Cl}}_{\mathbf{G}_p}^K(\langle w^h\rangle) = \langle {\rm root}_K(w^h)^{\nu_p({\rm exp}_K(w^h))}\rangle$$
and so
$${\rm{Cl}}_{\mathbf{G}_p}^K(\langle w^h\rangle) = \langle {\rm{root}}(w^h)^{h_{{\rm{root}}(w^h)}\nu_p({\rm exp}(w^h))}\rangle$$
by Lemma \ref{rek}(iii). 
In view of (\ref{pow1}), we get
$${\rm{Cl}}_{\mathbf{G}_p}^K(\langle w^h\rangle) = \langle u^{h_u\nu_p(he)}\rangle.$$
Since $h \mid p-1$, we get
$${\rm{Cl}}_{\mathbf{G}_p}^K(\langle w^h\rangle) = \langle u^{h_u\nu_p(e)}\rangle = \langle u^{h_up^s}\rangle$$
and it follows from (\ref{clv1}) that
$${\rm{Cl}}_{\mathbf{V}_p}^F(\langle w\rangle) = \bigcup_{i=0}^{h-1}\langle u^{h_up^s}\rangle w^i.$$
Thus we only need to show that
\begin{equation}
\label{clv2}
\bigcup_{i=0}^{h-1}\langle u^{h_up^s}\rangle w^i = \langle u^{d_p}\rangle.
\end{equation}

First, we show that 
\begin{equation}
\label{clv3}
\langle w^{h}\rangle \leq \langle u^{h_up^s}\rangle.
\end{equation}

Indeed, we have $w^{h} = u^{eh}$, so it suffices to show that $h_up^s$ divides $eh$. Clearly, 
$$\nu_p(h_up^s) = p^s = \nu_p(e) = \nu_p(eh).$$
Hence we only need to show that $h_u$ divides $eh$. 
Since $u^{eh}=w^h\in K$ and $u^{h_u}\in K$ by Lemma \ref{pk}, then also $u^{{\rm gcd}(eh,h_u)} \in K$. Now it follows from Lemma \ref{pk} that $h_u={\rm gcd}(eh,h_u)$, hence $h_u$ divides $eh$. 
Therefore (\ref{clv3}) holds.

Since $d_p$ divides $h_up^s$, we get $u^{h_up^s} \in \langle u^{d_p}\rangle$ and consequently $\langle u^{h_up^s}\rangle \leq \langle u^{d_p}\rangle$. On the other hand, 
$d_p$ also divides $e$ and so $w = u^e \in \langle u^{d_p}\rangle$. Therefore
$$\bigcup_{i=0}^{h-1}\langle u^{h_up^s}\rangle w^i \leq \langle u^{d_p}\rangle.$$

Conversely, let $v \in \langle u^{d_p}\rangle$. Then $v = u^{kd_p}$ for some $k \in \mathbb{Z}$. Let $x,y \in \mathbb{Z}$ be such that ${\rm{gcd}}(e,h_u) = xe + yh_u$. There exist some $z \in \mathbb{Z}$ and $i \in \{ 0,\ldots,h-1\}$ such that $kxp^s = zh + i$.
In view of (\ref{clv3}), we get
$$v = u^{kd_p} = u^{kxep^s + kyh_up^s} = (u^{h_up^s})^{ky}w^{kxp^s} = (u^{h_up^s})^{ky}(w^{h})^zw^i
\in \langle u^{h_up^s}\rangle w^i$$
and so $$\langle u^{d_p} \rangle \leq \bigcup_{i=0}^{h-1}\langle u^{h_up^s}\rangle w^{i}.$$ This establishes (\ref{clv2}) and we are done.
\end{proof}

\section*{Acknowledgements}

The first author acknowledges support from the Centre of Mathematics of the University of Porto, which is financed by national funds through the Funda\c c\~ao para a Ci\^encia e a Tecnologia, I.P., under the project with references UIDB/00144/2020 and  UIDP/00144/2020.

The second author acknowledges support from the Centre of Mathematics of the University of Porto, which is financed by national funds through the Funda\c c\~ao para a Ci\^encia e a Tecnologia, I.P., under the project with reference UIDB/00144/2020. 

The third author was supported by the Engineering and Physical Sciences Research Council, grant number 
EP/T017619/1.

\vspace{1cm}

{\sc Claude Marion, Centro de
Matem\'{a}tica, Faculdade de Ci\^{e}ncias, Universidade do
Porto, R. Campo Alegre 687, 4169-007 Porto, Portugal}

{\em E-mail address}: claude.marion@fc.up.pt

\bigskip

{\sc Pedro V. Silva, Centro de
Matem\'{a}tica, Faculdade de Ci\^{e}ncias, Universidade do
Porto, R. Campo Alegre 687, 4169-007 Porto, Portugal}

{\em E-mail address}: pvsilva@fc.up.pt

\bigskip

{\sc Gareth Tracey, Mathematics Institute, University of Warwick, Coventry CV4 7AL, U.K.}

{\em E-mail address}: Gareth.Tracey@warwick.ac.uk

\end{document}